\documentclass{amsart}
\usepackage[marginratio=1:1]{geometry}

\usepackage[hidelinks]{hyperref}
\usepackage{calc}
\newsavebox\CBox
\newcommand\hcancel[2][0.5pt]{%
  \ifmmode\sbox\CBox{$#2$}\else\sbox\CBox{#2}\fi%
  \makebox[0pt][l]{\usebox\CBox}%
  \rule[0.5\ht\CBox-#1/2]{\wd\CBox}{#1}}
\usepackage{blindtext}
\usepackage{color}
\usepackage{hyperref,url}
\usepackage{comment}
\usepackage{float}
\usepackage{hyperref}
\usepackage{enumerate}
\usepackage{enumitem}   
\usepackage{bbm}
\usepackage{cancel}
\usepackage{mathrsfs}
\usepackage{colonequals}
\usepackage{pdfpages}

\usepackage{amsmath,amsfonts,amsthm,bm}
\usepackage{mathrsfs}  
\usepackage{xcolor}
\usepackage{amsfonts}
\usepackage{amssymb}
\usepackage{amsmath}
\usepackage{mathtools}
\usepackage{mathrsfs}  

\usepackage[normalem]{ulem}

\numberwithin{equation}{section}
\usepackage[toc,page]{appendix}
\usepackage{tikz-cd}
\theoremstyle{definition}

\newtheorem{definicao}{Definition}[section]

\theoremstyle{plain}

\newtheorem{teorema}[definicao]{Theorem}
\newtheorem{suposicao}{Hypothesis}

\newtheorem{lema}[definicao]{Lemma}
\newtheorem{proposicao}[definicao]{Proposition}

\usepackage{scalerel}

\definecolor{roxo}{rgb}{0.44, 0.16, 0.39}
\definecolor{ao(english)}{rgb}{0.0, 0.5, 0.0}
\definecolor{dmagenta}{RGB}{139, 0, 139}
\definecolor{dgreen}{RGB}{0,90,0}
\definecolor{navy}{RGB}{0,0,128}

\usepackage{stackengine}

\definecolor{iblue}{RGB}{0, 35, 194}
\title[transition to chaos]{BV estimates for stationary and quasi-stationary densities as a means to establishing noise-induced chaos}
\author{Giuseppe Tenaglia}

\allowdisplaybreaks
\begin{document}


\maketitle

\begin{center}
  \small\itshape Imperial College London
\end{center}

\begin{abstract}
{In this paper we introduce a class of non uniformly expanding random dynamical system with additive noise and we prove a BV estimate between the stationary measure and the quasistationary measure of the system. Furthermore, we use these bounds to give precise estimates for the Lyapunov exponent of the system. }

\end{abstract}

\section{Introduction and motivation}
In deterministic one dimensional dynamics, it is common to observe systems which expand in a large portion of the phase space and contract in the immediate basin  of a  periodic sink, which has typically very small size. For these systems, the inclusion of additive noise with sufficiently strong size forces  trajectories to escape the sink and to spend more time in the expanding region, drastically changing its asymptotic statistics. In many cases a transition to chaos occurs, which means that beyond a critical noise amplitude the system displays a unique ergodic stationary measure equivalent to Lebesgue carrying positive Lyapunov exponent. 

Such phenomenon has been described, for example, in \cite{Bassols_Cornudella_2023}, where the authors study the bifurcation scenario associated to the random perturbation of the logistic map with parameter $a=3.83$, whose deterministic part has a globally attracting three periodic cycle  sitting in a trapping region. The authors suggest that, once the
noise amplitude is large enough to allow escape from this trapping
region, the statistics of the stationary measure are influenced
by those of a measure supported in the expanding region, called the
quasi-ergodic measure, which describes the limiting time
averages of trajectories conditioned on remaining in the expanding
region for a long time. This influence, the authors argue,
contributes to the transition to chaos. 

Even though their numerical results suggest that such a relation may
indeed exist, it is possible that, in their setting, a more suitable approximation of the stationary measure is given by the quasi-stationary measure associated with the expanding region, which
describes the distribution of trajectories conditioned on never leaving such region. To investigate this possibility, in this paper, we derive quantitative estimates
comparing the stationary measure of a randomly perturbed system with
the quasi-stationary measure for the dynamics killed upon entering a
possibly contracting region. We then use
these estimates to establish a transition to chaos in a simplified
version of the model studied in \cite{Bassols_Cornudella_2023}.

Indeed, we consider a piecewise $C^2$ uniformly expanding map $ f \colon X \to X$, where $X=\mathbb{S}^1,[0,1]$ and, given a point $x_0 \in X$, we modify $f$ in a $\delta$ neighbourhood of $x_0$ so that the resulting map $f_{\delta}$ is continuous and piecewise differentiable.

We then  consider the Markov chain $f_{\sigma,\delta}$ associated to the iterations 
\begin{equation}\label{iterations}
f^n_{\omega,\delta}(x) := f_{\theta^{n-1}(\omega),\delta} \circ \dots \circ  f_{\omega,\delta}(x) \qquad x \in X,
\end{equation}
where $\omega \in \Omega:= [-\sigma,\sigma]^{\mathbb{N}}$, $\theta \colon \Omega \to \Omega$ denotes the shift map on $\Omega$ and
\begin{equation}\label{map}
f_{\omega,\delta}(x):= f_{\delta}(x) + \omega_0\qquad  x \in X.
\end{equation}

Furthermore, we equip $\Omega$ with the product measure on $\mathbb{P}:= \left(\frac{\text{Leb}_{|[-\sigma,\sigma]}}{2\sigma}\right)^{\mathbb{N}}$ so that the iterations in \eqref{iterations} represent a Markov chain driven by IID uniform noise on $[-\sigma,\sigma]$.
Two are the measure associated to $f_{\sigma,\delta}$ of interest to us:
\begin{itemize}
\item the stationary measure $\mu_{\sigma,\delta}$, which can be interpreted as an average of the asymptotic statistics of the system and satisfies, for all  $n \ge 0$
\begin{align*}
\mu_{\sigma,\delta}(A)= \int_{J} \mathbb{P}\{\omega \colon f^n_{\omega,\delta}(x) \in A \}d\mu_{\sigma,\delta}(x) \qquad \forall A \subset \mathcal{B}(J);
\end{align*}
\item the quasi-stationary measure $\nu_{\sigma,\delta}$, which encodes the statistics of orbits that never leave the expanding region, and satisfies, for all $n \ge 0$
\begin{align*}
\nu_{\sigma,\delta}(A)= \int_{J} \frac{\mathbb{P}\{\omega \colon f^n_{\omega,\delta}(x) \in A,\,\text{and}\, f^{i}_{\omega}(x) \in J \setminus B_{\delta}(x_0) \,\,\forall i \le n-1 \}}{\int_J \mathbb{P}\left\{\omega \colon f^{i}_{\omega,\delta}(x) \in J \setminus B_{\delta}(x_0) \,\,\forall i \le n-1 \right\}d\nu_{\sigma,\delta}(x)}d\nu_{\sigma,\delta}(x) \qquad \forall A \subset \mathcal{B}(J).
\end{align*}
\end{itemize}
Under the assumption of existence and uniqueness of $\mu_{\sigma,\delta}$ and $\nu_{\sigma,\delta}$, our theoretical result, Theorem \ref{thm1} shows that such measures are absolutely continuous to Lebesgue and thier densities, respectively $\rho_{\sigma,\delta}$ and $q_{\sigma,\delta}$ are functions of bounded variation. It then estimates the BV norm of the difference between a fraction of $\rho_{\sigma,\delta}$ and  $q_{\sigma,\delta}$ in terms of two key quantities:
\begin{itemize}
\item the ratio  $\frac{\delta}{\sigma}$ between the size of the hole and the size of the perturbation;
\item the length of the gap time $k_{\sigma,\delta}$  between two visits to $B_{\delta}(x_0)$, and the rate of expansion of orbits during this gap time. More in details, the key quantity $k_{\sigma,\delta}$ is defined as 
\begin{equation}\label{gaptime}
k_{\sigma,\delta}:= \min \left\{ \min_{y \in B_{\delta}(x_0)} \min \{n \ge 1 \colon \mathbb{P}\{f^n_{\omega,\delta}(x) \in B_{\delta}(x_0) \}>0 \}, \left\lfloor |\log\delta| \right\rfloor \right\}.
\end{equation}
\end{itemize}
The quantitative result stated in Theorem \ref{thm1} can be summarized as follows:
\begin{itemize}
\item if $x_0$ is a fixed point for $f_{\delta}$, then
$$
||\rho_{\sigma,\delta}-q_{\sigma,\delta} ||_{BV} = O\left( \delta + \frac{\delta}{\sigma}\right) \qquad \text{as}\,\, \,\,\delta,\frac{\delta}{\sigma} \to 0
$$ 

\item if $x_0$ is not a fixed point, then $\rho_{\sigma,\delta}$ can be written as
\begin{align*}
\rho_{\sigma,\delta}= \frac{u_{\sigma,\delta} + R_{\sigma,\delta}}{1+|R_{\sigma,\delta}|_{L^1}},
\end{align*}
where $R_{\sigma,\delta}$ vanishes in $B_{\delta}(x_0)$, satisfies $|R_{\sigma,\delta}|_{L^1} = O\left(\delta k_{\sigma,\delta} \right)$,  $|R_{\sigma,\delta}|_{BV}= O\left(\frac{\delta k_{\sigma,\delta}}{\sigma}\right)$ and
\begin{align*}
||u_{\sigma,\delta}-q_{\sigma,\delta}||_{BV} = O\left(\delta k_{\sigma,\delta} + \frac{\delta}{\sigma}\frac{1}{\lambda^{k_{\sigma,\delta}}}\right) \qquad \text{as} \; \;\delta,\frac{\delta}{\sigma \lambda^{k_{\sigma,\delta}}} \to 0,
\end{align*}
where  $\lambda := \inf_{x \in X} |f'(x)|>1$.
\end{itemize}

The main difficulty in obtaining such a result is that $\mu_{\sigma,\delta}$ and $\nu_{\sigma,\delta}$ are in principle very different. Indeed, the measure $\nu_{\sigma,\delta}$ describes the statistics
of orbits conditioned on never entering 
$B_{\delta}(x_0)$ and does not depend on how $f_\delta$ is
defined on $B_{\delta}(x_0)$. By contrast, $\mu_{\sigma,\delta}$ also
depends on the orbits that enter $B_{\delta}(x_0)$. In general, their
contribution can substantially influence the shape of the invariant
measure, since orbits can spend a long time in this region.
 To address these difficulties, in the first item of
Theorem~\ref{thm1} we use a direct computation that exploits standard
stability results for expanding maps with noise regularisation from
\cite{galatolo2022statistical}. The proof of the second item is more
involved and relies on the analysis of an auxiliary operator inspired
by \cite{Young1992-vy} and on  ideas from \cite{blublu}.

To the best of our knowledge, Theorem~\ref{thm1} is the first result
to provide a quantitative approximation in the $BV$ norm of a
stationary measure by a quasi-stationary measure when the underlying
system is non-uniformly expanding. Other results concerning the
approximation of closed systems by open systems
\cite{Bahsoun_2012,Bahsoun_2013,GONZÁLEZ-TOKMAN_HUNT_WRIGHT_2011,gonzalez2021lyapunov},
in both random and deterministic settings, rely on uniform expansion
to control the perturbation using classical Lasota--Yorke arguments.
In our setting, these techniques cannot be applied directly, since
the underlying system may have critical points.

 More in general, there is a wide literature on how to treat open systems with holes in both deterministic and random system and many different topics are covered: constructions of Young towers \cite{Demers_2006,Demers2005-zq,DEMERS_TODD_2021},  existence of a thermodynamic formalism for open systems \cite{atnip2021thermodynamic,cornudella2024conditioned}, studying conditioned dynamics \cite{CASTRO2024104364,castro2022lyapunov}. Furthermore,  especially after the pioneering work of Keller-Liverani \cite{Keller1999} on stochastic stability of perturbed operators, a lot of effort has been put into studying dynamics of piecewiese expanding maps with small holes, either in the deterministic  \cite{FERGUSON_POLLICOTT_2012,LIVERANI2003385,Keller2009-sh}
and in the random setting \cite{Atnip2022PerturbationFF}.

In Theorem~\ref{thm2}, we apply the $BV$ estimates of
Theorem~\ref{thm1} to establish  transition to chaos for a class of systems that provides a toy model for the randomly perturbed logistic map studied in \cite{Bassols_Cornudella_2023}. These systems are random perturbations of maps of the form 
\[
f_\delta(x)=
\begin{cases}
f(x), & x\notin B_\delta(x_\delta),\\
g_\delta(x), & x\in B_\delta(x_\delta),
\end{cases}
\]
where $f$ is a uniformly expanding $C^2$ circle map, $x_\delta$ is a periodic point of $f$, and $g_\delta$ is chosen so that $f_\delta$
is continuous, $g_\delta(x_\delta)=f(x_\delta)$, 
$g_\delta'(x_\delta)=0$ and $g(B_{\delta}(x_\delta))=f(B_{\delta}(x_\delta))$. 
In particular, $x_\delta$ is a super-attracting periodic point of
$f_\delta$, and the stated hypotheses are compatible with its periodic
orbit having a basin of attraction of full Lebesgue measure. We then consider the Markov chain $f_{\sigma,\delta}$ obtained from the random iterations, as in \eqref{iterations},  of the maps
$f_{\omega,\delta}(x)=f_\delta(x)+\omega_0,$
with $\omega\in\Omega=[-\sigma,\sigma]^{\mathbb N}$, endowed with the product measure $\mathbb P
:=\left(\frac{\operatorname{Leb}|_{[-\sigma,\sigma]}}{2\sigma}\right)^{\mathbb N}.$  We consider separately the cases in which $x_\delta$ is a fixed
point, has a fixed period greater than one, or has a period
$p_\delta\asymp\log(1/\delta)$.

In this simplified setting, Theorem~\ref{thm2} uses the estimates
of Theorem~\ref{thm1} to quantify the amount of noise sufficient
to obtain a transition to chaos, understood as a change of sign of
the Lyapunov exponent. It also provides an interpretation of the
transition as a consequence of the noise giving visibility to the
transient statistics encoded in the quasi-stationary measure.
Although an analytical proof may be difficult to obtain, we believe
that a similar phenomenon contributes to the transition to chaos in
more complicated systems, such as the random logistic map studied in
\cite{Bassols_Cornudella_2023}.

Quantitative results in the context of transition to chaos have been previously obtained in \cite{lian,blublu} for predominantly expanding multimodal circle maps with additive noise. In \cite{lian}, the authors  estimate the amount of the noise needed to destroy a sink of period $1$, whilst in \cite{blublu}, the author shows the amount of noise needed to destroy sinks of very high period, obtaining estimates similar to ours.  
It is worth emphasizing that Theorem \ref{thm2} and the results in \cite{lian,blublu} deal with random perturbation of deterministic maps with sinks and critical points and, even if we are in the setting of small noise amplitude $\sigma$, the noise is still large enough for the transition to chaotic regime to happen. 
Other works dealing with small noise regime typically study random perturbation of maps with a good asymptotic regime, such as smooth or piecewise smooth interval maps \cite{viana1997stochastic}, maps with indifferent fixed points \cite{svanstri} ,and nonuniformly expanding maps with critical points  \cite{Baladi1996,ASENS_2002_4_35_1_77_0}
 
\section{statement of the results}
\subsection{Hypotheses and main result}
Let $X$ be either the unit circle or the interval $[0,1]$. For
$I\subset X$, let $L^1(I)$ denote the space of Lebesgue-integrable functions on $I$,
equipped with the norm
\[
\|\varphi\|_{L^1(I)} := \int_I |\varphi(x)|\,dx.
\] Furthermore, let $BV(I)$ denote the space of functions of bounded variation on $I$, equipped with the norm
\[
\|\varphi\|_{BV(I)}
:= \|\varphi\|_{L^1(I)}+\operatorname{Var}_I(\varphi),
\]
where $\operatorname{Var}_I(\varphi)$ denotes the total variation of
$\varphi$ on $I$:
\[
\operatorname{Var}_I(\varphi)
:=
 \sup_{\substack{n\geq 1\\ x_0<\cdots<x_n,\;x_j\in I}}
\sum_{j=1}^{n}
\bigl|\widetilde\varphi(x_j)-\widetilde\varphi(x_{j-1})\bigr|.
\]
We also define
\[
BV_0(I)
:=
\left\{
\varphi\in BV(I)
:
\int_I \varphi\,dx=0
\right\}.
\]
Let $f\colon X\to X$ be a continuous  piecewise $\mathcal{C}^2$ expanding map, and let $\mathcal{L}\colon BV(X)\to BV(X)$ denote its transfer
operator, defined by
\[
\mathcal{L}\varphi(x)
=
\sum_{y\in f^{-1}(x)}
\frac{\varphi(y)}{|f'(y)|},
\qquad \text{for Lebesgue-a.e. }x\in X,\,\, \varphi \in BV(X).
\]
Our first assumption concerns the turning points of $f$ and the decay of the iterates of $\mathcal{L}$ on $BV_0(X)$.

\begin{suposicao}\label{(Z)}
The map $f$ has no periodic turning points. Furthermore, there exist constants $C>0$ and $\theta\in(0,1)$ such that,
for every $n\geq 0$ and every $\varphi\in BV_0(X)$,
\[
\|\mathcal{L}^n\varphi\|_{BV(X)}
\leq C\theta^n\|\varphi\|_{BV(X)}.
\]
\end{suposicao}

Fix a sufficiently small $\delta>0$ and a point $x_0\in X$. Let $B_\delta(x_0)$ be the ball or radius $\delta$ around $x_0$, and  let
\[
g\colon \overline{B_\delta(x_0)}\to X
\]
be a continuous piecewise $\mathcal{C}^1$ map such that
$g'(x)\neq 0$ for Lebesgue-almost every $x\in B_\delta(x_0)$
and satisfies
\begin{equation}\label{r1}
g(x_0\pm\delta)=f(x_0\pm\delta).
\end{equation}
We define the modified map $f_\delta\colon X\to X$ by
\begin{equation}\label{defi}
f_\delta(x)=
\begin{cases}
g(x), & x\in B_\delta(x_0),\\
f(x), & x\in X\setminus B_\delta(x_0).
\end{cases}
\end{equation}
We assume that $f$ and $f_\delta$ admit a common invariant interval
that remains invariant under sufficiently small additive perturbations.

\begin{suposicao}\label{(Z')}
For all sufficiently small $\sigma,\delta>0$, there exists an interval
$J_{\sigma,\delta}\subset X$ such that
\[
f(J_{\sigma,\delta})+[-\sigma,\sigma]
\subset J_{\sigma,\delta}
\qquad\text{and}\qquad
f_\delta(J_{\sigma,\delta})+[-\sigma,\sigma]
\subset J_{\sigma,\delta},
\]
where  for a set $A\subset X$, $$A+[-\sigma,\sigma]:=\{x+t:x\in A,\ t\in[-\sigma,\sigma]\}.$$
\end{suposicao}
Note that,  if $f$ satisfies both Hypotheses \ref{(Z)} and \ref{(Z')}, then, up to possibly increasing $C,$ we have 
\begin{equation}\label{maybeused}
\|\mathcal{L}^n\varphi\|_{BV(J_{\sigma,\delta})}
\leq C\theta^n\|\varphi\|_{BV(J_{\sigma,\delta})},\quad  \forall \phi \in BV_0(J_{\sigma,\delta})
\end{equation}
for all $\sigma,\delta>0$ sufficiently small.

Given $\sigma>0$, consider the probability space $\Omega:= [-\sigma,\sigma]^{\mathbb{N}}$ endowed with the product Borel sigma algebra $\mathcal{F}$ and  the probability measure $\mathbb{P}:=\left(\frac{\text{Leb}_{|[-\sigma,\sigma]}}{2\sigma}\right)^{\mathbb{N}}$. Let
\begin{align*}
\theta \colon \Omega \to \Omega, \qquad 
\theta(\{\omega_i\}_{i \ge 0}) = \{\omega_{i+1}\}_{i \ge 0},
\end{align*}
denote the shift map on $\Omega$, so that $(\Omega, \mathbb{F},\mathbb{P},\theta)$ is an ergodic dynamical system.

By Hypothesis \ref{(Z')}, the family of maps
\begin{align*}
f_{\omega,\delta} &\colon J_{\sigma,\delta} \to J_{\sigma,\delta}\\
f_{\omega,\delta}(x)&= f_{\delta}(x)+\omega_0
\end{align*}
induces the Markov chain $f_{\sigma,\delta}$ via the iterations
\begin{align*}
f^n_{\omega,\delta}(x)= f_{\theta^{n-1}(\omega),\delta} \circ \dots \circ f_{\omega,\delta}(x) \qquad \forall x \in J_{\sigma,\delta},\,\,\omega \in \Omega.
\end{align*}

The annealed transfer operator associated to the process $f_{\sigma,\delta}$  is the operator 
\begin{align*}
\mathcal{L}_{\sigma,\delta} &\colon BV(J_{\sigma,\delta}) \to BV(J_{\sigma,\delta})\\
\mathcal{L}_{\sigma,\delta}\varphi(x)
&=
\frac{1}{2\sigma}
\int_{-\sigma}^{\sigma}
\mathcal{L}_{\delta}\varphi(x-t)\,dt,
\qquad x\in J_{\sigma,\delta},
\end{align*}
where $\mathcal{L}_{\delta}$ is transfer operators associated to the map  $f_{\delta}$.
Recall that, if $\rho_{\sigma,\delta}$ is a fixed point for $\mathcal{L}_{\sigma,\delta}$, then the measure $\mu_{\sigma,\delta}= \rho_{\sigma,\delta} dx$ is stationary measure for $f_{\sigma,\delta}$, hence satisfies
\begin{align*}
\mu_{\sigma,\delta}(A) = \int_{\mathbb{S}^1}\mathbb{P}\{f_{\omega,\delta}(x) \in A \} d\mu_{\sigma,\delta}(x),
\end{align*}
for all Borel sets $A$.
Furthermore, $\mu_{\sigma,\delta}$ is said to be ergodic if all the almost surely invariant sets for $f_{\sigma,\delta}$ have either $\mu$ measure $0$ or $1$. In this case, Birkhoff's ergodic theorem states that
\begin{align*}
\lim_{n\to\infty}\frac{1}{n}\sum_{i=0}^{n-1}
\phi(f^i_{\omega,\delta}(x))
=
\int_{J_{\sigma,\delta}}\phi(y)\,d\mu_{\sigma,\delta}(y)
\qquad
\mathbb{P}\otimes\mu_{\sigma,\delta}\text{-a.s.},
\end{align*}
for every observable $\phi$ that is integrable with respect to $\mu_{\sigma,\delta}.$ For a more detailed background on Markov chains, see e.g. 
 \cite{norris1998markov,hairer2010convergence}. 
 In the following Hypothesis, we require that the process $f_{\sigma,\delta}$ admits unique ergodic measure supported in $J_{\sigma,\delta}$.
\begin{suposicao} \label{(A)} 
 \item  The Markov chain $f_{\sigma,\delta}$ admits a unique ergodic
stationary probability measure $\mu_{\sigma,\delta}$ supported in
$J_{\sigma,\delta}$.
\end{suposicao}
The smoothing argument in the proof of Proposition \ref{finalestimateeee} shows that $\mu_{\sigma,\delta}$ has a density $\rho_{\sigma,\delta},$ which is a function of bounded variation. In particular, it is the unique fixed point of the transfer operator $\mathcal{L}_{\sigma,\delta} \colon BV(J_{\sigma,\delta}) \to BV(J_{\sigma,\delta})$.
For simplicity, write $H_\delta:=B_\delta(x_0)$. We define the
conditioned transfer operator associated with the dynamics killed upon entering $H_\delta$ by
\begin{align}
\mathcal{R}_{\sigma,\delta}
&\colon BV(J_{\sigma,\delta})\to BV(J_{\sigma,\delta}), \notag\\
\mathcal{R}_{\sigma,\delta}\varphi
&:=
\mathcal{L}_{\sigma,\delta}
\bigl(\mathbf{1}_{J_{\sigma,\delta}\setminus H_\delta}\varphi\bigr).
\label{Condto}
\end{align}
It is well known \cite{CASTRO2024104364} that if $q_{\sigma,\delta}$ is a positive eigenfunction with real eigenvalue for  $\mathcal{R}_{\sigma,\delta}$, then the measure $\nu_{\sigma,\delta}= q_{\sigma,\delta} dx$   satisfies
\begin{align*}
\nu_{\sigma,\delta}(A)
=
\frac{
\displaystyle
\int_{J_{\sigma,\delta}}
\mathbb{P}\left\{
\omega\colon
f_{\omega,\delta}^n(x)\in A
\ \text{and}\
f_{\omega,\delta}^i(x)\in J_{\sigma,\delta}\setminus H_\delta
\ \text{for all } 0\leq i\leq n-1
\right\}
\,d\nu_{\sigma,\delta}(x)
}{
\displaystyle
\int_{J_{\sigma,\delta}}
\mathbb{P}\left\{
\omega\colon
f_{\omega,\delta}^i(x)\in J_{\sigma,\delta}\setminus H_\delta
\ \text{for all } 0\leq i\leq n-1
\right\}
\,d\nu_{\sigma,\delta}(x)
},
\end{align*}
for all measurable  $A\subset J_{\sigma,\delta}.$ In such case, we say that $\nu_{\sigma,\delta}$ is a quasi-stationary measure for $f_{\sigma,\delta}$.The next proposition uses uniform expansion outside $H_\delta$ and
the absence of periodic turning points of $f$ to establish the
estimates on $\mathcal R_{\sigma,\delta}$ needed in the proof of our main result, Theorem~\ref{thm1}.
\begin{proposicao}\label{qsm}
The following holds for all $\sigma,\delta$ small enough:
\begin{itemize}
\item The operator $\mathcal{R}_{\sigma,\delta}$ in equation \eqref{Condto}
satisfies a Lasota-Yorke inequality, which means that there exists
$A,B \ge 1$, $\gamma<1$ such that
\begin{equation}\label{Lasota}
\operatorname{Var}_{J_{\sigma,\delta}}\left(\mathcal{R}_{\sigma,\delta}^n(\phi)\right)
\le A\gamma^n \operatorname{Var}_{J_{\sigma,\delta}}(\phi)
+ B|\phi|_{L^1(J_{\sigma,\delta})}.
\end{equation}

\item We have the estimate
\begin{equation}\label{K-L}
\|\mathcal{R}_{\sigma,\delta}-\mathcal{L}_0\|_{BV(J_{\sigma,\delta})\to L^1(J_{\sigma,\delta})}
:=
\sup_{\|\varphi\|_{BV(J_{\sigma,\delta})}\leq 1}
\|(\mathcal{R}_{\sigma,\delta}-\mathcal{L}_0)\varphi\|_{L^1(J_{\sigma,\delta})} = O(\sigma+\delta).
\end{equation}

\item The operator $\mathcal{R}_{\sigma,\delta}$ is quasi-compact and
its spectral radius $\lambda_{\sigma,\delta}$ satisfies
\begin{equation}\label{spectralradiusstability}
|1-\lambda_{\sigma,\delta}| = O(\delta) \qquad \text{as } \delta \to 0.
\end{equation}
Furthermore, associated to $\lambda_{\sigma,\delta}$ there is an unique positive eigenfunction $q_{\sigma,\delta}$ with uniformly bounded
$BV(J_{\sigma,\delta})$ norm, i.e. there exists a constant $C>0$ such
that for all $\sigma,\delta$ small enough
\begin{equation}\label{qsmbv}
|q_{\sigma,\delta}|_{BV(J_{\sigma,\delta})} \le C.
\end{equation}

\item Let $\rho_0$ be the invariant measure for $f$. Then
\begin{equation}\label{convergence}
|\rho_0-q_{\sigma,\delta}|_{L^1(J_{\sigma,\delta})}
= O\left(\sigma |\log(\sigma)|+\delta |\log(\delta)|\right)
\end{equation}
\end{itemize}
\end{proposicao}
\begin{proof}
Let $\mathcal L_\sigma$ denote the  transfer operator
associated with $f_{\sigma,0}$, which is just the random perturbation of $f$. By Hypothesis~\ref{(Z)} and
\cite[Chapter~3, Section~5, equation~(12)]{viana1997stochastic},
there exist constants $A,B>0$ and $\eta\in(0,1)$, independent of
$\sigma,\delta$, such that
\begin{align*}
\operatorname{Var}_{J_{\sigma,\delta}}
  (\mathcal L_\sigma^n\phi)
\leq
A\eta^n\operatorname{Var}_{J_{\sigma,\delta}}(\phi)
+B\|\phi\|_{L^1(J_{\sigma,\delta})}.
\end{align*}
Let $\|\cdot\|_{L^\infty(J_{\sigma,\delta})}$ denote the supremum norm on
$J_{\sigma,\delta}$. By the above inequality, we have
\[
\begin{aligned}
\|\mathcal R_{\sigma,\delta}^n\phi\|_{L^\infty(J_{\sigma,\delta})}
&\leq
\|\mathcal L_\sigma^n|\phi|\|_{L^\infty(J_{\sigma,\delta})}\\
&\leq
\operatorname{Var}_{J_{\sigma,\delta}}
  (\mathcal L_\sigma^n|\phi|)
+\frac{1}{|J_{\sigma,\delta}|}
  \|\mathcal L_\sigma^n|\phi|\|_{L^1(J_{\sigma,\delta})}\\
&\leq
C\eta^n\operatorname{Var}_{J_{\sigma,\delta}}(\phi)
+C\|\phi\|_{L^1(J_{\sigma,\delta})},
\end{aligned}
\]
for some $C>0$ and $\eta\in(0,1)$, independent of $\sigma,\delta$. Arguing as in the proof of \cite[Lemma~62]{galatolo2022statistical}
and averaging over the noise, we obtain
\[
\begin{aligned}
\operatorname{Var}_{J_{\sigma,\delta}}
  (\mathcal R_{\sigma,\delta}^n\phi)
&=
\operatorname{Var}_{J_{\sigma,\delta}}
  \left(
    \mathcal L_\sigma\left(
      \mathbf 1_{J_{\sigma,\delta}\setminus H_\delta}
      \mathcal R_{\sigma,\delta}^{n-1}\phi
    \right)
  \right)
\\
&\leq
\lambda^{-1}
\operatorname{Var}_{J_{\sigma,\delta}}
  (\mathcal R_{\sigma,\delta}^{n-1}\phi)
+C\|\mathcal R_{\sigma,\delta}^{n-1}\phi
  \|_{L^\infty(J_{\sigma,\delta})}
\\
&\leq
\lambda^{-1}
\operatorname{Var}_{J_{\sigma,\delta}}
  (\mathcal R_{\sigma,\delta}^{n-1}\phi)
+C\eta^{n-1}\operatorname{Var}_{J_{\sigma,\delta}}(\phi)
+C\|\phi\|_{L^1(J_{\sigma,\delta})}.
\end{aligned}
\]

Iterating this inequality inductively proves \eqref{Lasota}.

Estimate \eqref{K-L} follows from the result in \cite[Example 3.1]{Liverani2004InvariantMA} and \cite[Lemma 7.2]{LIVERANI2003385}. The $BV$ estimate in \eqref{qsmbv} has been derived in \cite{GONZÁLEZ-TOKMAN_HUNT_WRIGHT_2011} as a consequence of Lasota-Yorke inequality and the Keller-Liverani stability result in \cite{Keller1999}. The estimate in the spectral radius is equation $2.2$ in \cite{Keller2009-sh}. Finally, the rate of convergence in \eqref{convergence} follows from \eqref{K-L} and  \cite[Proposition 38]{galatolo2022statistical}.
\end{proof}
Our last assumption concerns the behaviour of orbits starting in $H_{\delta}$. Let $k_{\sigma,\delta}$ denote the minimal amount of time needed for a point $x \in H_{\delta}$  to return to $H_{\delta}$ as defined in equation \eqref{gaptime}. We require that no orbit intersects the set of non-differentiability points of $f$ during the gap time.
\begin{suposicao} \label{(C)}
Let 
$$\mathcal{C}_{\delta}:= \{x \in X \setminus H_{\delta}  \colon \text{f is not differentiable in x} \}$$.

For all $\omega \in \Omega$ and $j \le k_{\sigma,\delta}$.
\begin{align*}
f^j_{\omega,\delta}(H_{\delta}) \cap \mathcal{C}_{\delta} = \emptyset.
 \end{align*}
\end{suposicao}
The following is our main  result.
\begin{teorema}\label{thm1}
Let $f_{\sigma,\delta}$ be the Markov chain generated by the iterations in
\eqref{iterations} and suppose it satisfies Hypotheses
\ref{(Z)}--\ref{(C)}. Then,
\begin{itemize}
\item if $x_0$ is a fixed point for $f_\delta$, then
\begin{equation}\label{fixedpest}
\|\rho_{\sigma,\delta}-q_{\sigma,\delta}\|_{BV(J_{\sigma,\delta})}
= O\left(\delta+\frac{\delta}{\sigma}\right)
\qquad
\text{as } \delta,\frac{\delta}{\sigma}\to 0;
\end{equation}

\item if $x_0$ is not a fixed point, let $k_{\sigma,\delta}$ be
defined as in \eqref{gaptime}. Then, $\rho_{\sigma,\delta}$ can be
written as
\begin{equation}\label{cooleq}
\rho_{\sigma,\delta}
=
\frac{u_{\sigma,\delta}+R_{\sigma,\delta}}
{1+\|R_{\sigma,\delta}\|_{L^1(J_{\sigma,\delta})}},
\end{equation}
where $R_{\sigma,\delta}$ is supported in
$J_{\sigma,\delta}\setminus B_\delta(x_0)$ and satisfies
\begin{align*}
\|R_{\sigma,\delta}\|_{L^1(J_{\sigma,\delta})}
&\leq C\delta k_{\sigma,\delta},\\
\|R_{\sigma,\delta}\|_{BV(J_{\sigma,\delta})}
&\leq C\frac{\delta k_{\sigma,\delta}}{\sigma}
\end{align*}
for some $C\geq 1$, and
\begin{equation}\label{cooleq2}
\|u_{\sigma,\delta}-q_{\sigma,\delta}\|_{BV(J_{\sigma,\delta})}
=
O\left(
\delta k_{\sigma,\delta}
+\frac{\delta}{\sigma}\frac{1}{\lambda^{k_{\sigma,\delta}}}
\right)
\qquad
\text{as }
\delta,\frac{\delta}{\sigma\lambda^{k_{\sigma,\delta}}}\to 0,
\end{equation}
where $\lambda:=\inf_{x\in X}|f'(x)|>1$.
\end{itemize}
\end{teorema}
\section{Applications: transition to chaos}

In this section, we first introduce a class of maps satisfying the assumptions of Theorem~\ref{thm1}, and then use this result to establish noise-induced chaos. This class consists of random perturbations of circle maps possessing a periodic sink to which almost every trajectory converges. It provides a toy model for the behaviour observed in the randomly perturbed logistic map studied in~\cite{Bassols_Cornudella_2023}.

\begin{definicao}\label{admissible}
We say that a family of Markov chains $f_{\sigma,\delta}$ is
admissible if it is obtained by additive random perturbations, as
in \eqref{iterations}, of maps $f_\delta$ having the form
\eqref{defi}, where $f$ and
$g\colon B_\delta(x_0)\to\mathbb S^1$ satisfy the following
conditions:
\begin{itemize}
\item[a] The map $f$ is $C^2$ and uniformly expanding on
$\mathbb S^1$;

\item[b] The map $f_\delta$ satisfies the following properties:
\begin{enumerate}
\item[1]\label{a} The point $x_0\in\mathbb S^1$ is a periodic point of $f$
of period $p_\delta$, and either $p_\delta=p$ is independent of
$\delta$, or 
\[
p_\delta\asymp\log(1/\delta)
\qquad\text{as }\delta\to0.
\]

\item[2]\label{b} The map $g$ is piecewise $ C^1$, satisfies
\eqref{r1}, and
\begin{equation}\label{r2}
g(B_\delta(x_0))=f(B_\delta(x_0));
\end{equation}

\item[3]\label{c} We have $g(x_0)=f(x_0)$ and $g'(x_0)=0$. Furthermore, 
there exists $l\geq 1$ and $c>0$ such that
\[
|g'(x)|\ge c |x-x_0|^l
\qquad \text{for Lebesgue-a.e. } x\in B_\delta(x_0),
\]
uniformly in $\sigma,\delta.$
\item[4]\label{d} Let $V_\delta\subset B_\delta(x_0)$ be the maximal set
containing $x_0$ such that $f_\delta^{p_\delta}(V_\delta)\subset V_\delta.$ Then $|g'(x)|>1$ for every
$x\in B_\delta(x_0)\setminus V_\delta$. In, addition, if $p_{\delta}\asymp \log(1/\delta),$  there exists $d_*>0$, independent of $\delta$, such that
\begin{equation}\label{boundedaway}
d_{\mathbb S^1}\bigl(f^j_{\delta}(V_{\delta}),x_0\bigr)\geq d_*
\qquad\text{for every }1\leq j<p_\delta.
\end{equation}
\end{enumerate}
\end{itemize}
\end{definicao}

Let $f_{\sigma,\delta}$ be an admissible family. For $x\in\mathbb S^1$ and $\omega \in \Omega$, we define the Lyapunov exponent at $(x,\omega)$ as 
\[
\xi_{\sigma,\delta}(x,\omega)
:=
\liminf_{n\to\infty}
\frac{1}{n}\log\left|df_{\omega,\delta}^n\right|.
\]
As shown in Theorem~\ref{thm2}, under the stated conditions on
$\sigma$ and $\delta$, the chain $f_{\sigma,\delta}$ admits a unique
ergodic measure $\mu_{\sigma,\delta}
=\rho_{\sigma,\delta}\,dx$. Its Lyapunov exponent is then  given by
\[
\xi_{\sigma,\delta}
=
\int_{\mathbb S^1}
\log|f_\delta'(x)|\,d\mu_{\sigma,\delta}(x)
=
\int_{\mathbb S^1}
\log|f_\delta'(x)|\rho_{\sigma,\delta}(x)\,dx.
\]

Note that, in the absence of noise, $f_\delta$ may display trivial
asymptotic dynamics, in the sense that the orbit of Lebesgue-almost
every point may converge to the periodic orbit of $x_0$. Since the map is uniformly
expanding outside $B_\delta(x_0)$, we expect that, for $\delta$ small
and $\sigma\gg\delta$, the system $f_{\sigma,\delta}$ admits a unique
ergodic measure on $\mathbb{S}^1$ with positive Lyapunov exponent.
Theorem~\ref{thm2} confirms this intuition and uses
Theorem~\ref{thm1} to estimate the noise amplitude needed to unlock
chaos.

More specifically, we estimate the difference $r_{\sigma,\delta}:=\xi_{\sigma,\delta}-\xi_0,$
where $\xi_0$ is the Lyapunov exponent of $f$. In particular, we
obtain conditions on the noise amplitude under which
$\xi_{\sigma,\delta}$ is sufficiently close to $\xi_0>0$ to ensure that
$\xi_{\sigma,\delta}>0$, establishing a transition to chaos.

\begin{teorema}\label{thm2}
Let $f_{\sigma,\delta}$ be an admissible family .  Then Hypotheses \ref{(Z)}-\ref{(C)} are satisfied provided $\sigma,\delta$ are small enough and satisfy $\sigma \lambda^{p_{\delta}-1} \ge 2\delta$, where $\lambda:=\inf_{x\in\mathbb S^1}|f'(x)|>1.$

As a consequence, we have that
\begin{itemize}
\item if $p_\delta=1$, then
\begin{equation}\label{est1}
\xi_{\sigma,\delta}
\geq \xi_0
-O\left(
\frac{\delta}{\sigma}
+\sigma|\log\sigma|
+\delta|\log\delta|
\right);
\end{equation}

\item if $p_\delta=p>1$, then
\begin{equation}\label{est2}
\xi_{\sigma,\delta}
\geq \xi_0
-O\left(
\delta
+\frac{\delta}{\sigma\lambda^{p}}
+\sigma|\log\sigma|
+\delta|\log\delta|
\right).
\end{equation}

\item If $p_\delta\asymp\log(1/\delta)$, then, choosing
$\sigma=\frac{\delta}{\log(1/\delta)}\ll\delta$, we have
\begin{equation}\label{finalest}
\xi_{\frac{\delta}{\log(1/\delta)},\delta}
\geq \xi_0
-O\left(\delta^{\alpha}\log(1/\delta)\right)
\qquad\text{as }\delta\to0,
\end{equation}
for some $\alpha>0$.
\end{itemize}
\end{teorema}

\section{Preliminaries on stochastic stability}
In this section we recall some known facts in stochastic stability which will be useful in the proof of Theorem \ref{thm1}. 
Let $f_{\sigma,\delta}$ be a random dynamical system constructed with the procedure outlined in Section $2$, satisfying Hypotheses \ref{(Z')}-\ref{(C)}, and  let  $\mathcal{T}_{\sigma,\delta}$ be a family of operators associated to $f_{\sigma,\delta}$. Suppose that this family admits unique fixed point $t_{\sigma,\delta}  \in BV(J_{\sigma,\delta})$. The following Proposition is taken from \cite{galatolo2022statistical} and shows how to estimate the difference in BV norm between $t_{\sigma,\delta}$ and the quasi-stationary density $q_{\sigma,\delta}$, as introduced in Proposition \ref{qsm}.
\begin{proposicao}\label{Bvdiff}
Let $\mathcal{T}_{\sigma,\delta},\,\, BV(J_{\sigma,\delta}) \to BV(J_{\sigma,\delta}) $ be defined as above  and let $\mathcal{R}_{\sigma,\delta}$ be the conditioned transfer operator as defined in \eqref{Condto}. Let  $q_{\sigma,\delta}$ be its maximal eigenfunction of eigenvalue $\lambda_{\sigma,\delta}$ and let 
\begin{align*}
\widetilde{\mathcal{R}}_{\sigma,\delta}:= \frac{\mathcal{R}_{\sigma,\delta}}{\lambda_{\sigma,\delta}}.
\end{align*}
Suppose that $|t_{\sigma,\delta}|_{L^1(J_{\sigma,\delta})}=|q_{\sigma,\delta}|_{L^1(J_{\sigma,\delta})}=1$ and, in addition, that $I-\mathcal{T}_{\sigma,\delta}$ is invertible with bounded inverse in $BV_0(J_{\sigma,\delta})$.
Then
\begin{equation}\label{BVtoBVest}
\begin{aligned}
&||t_{\sigma,\delta}- q_{\sigma,\delta}||_{BV(J_{\sigma,\delta})} \\
&\le \left|\left|\left(I-\mathcal{T}_{\sigma,\delta}\right)^{-1}_{|BV_0(J_{\sigma,\delta})} \right|\right|_{BV_0(J_{\sigma,\delta}) \to BV_0(J_{\sigma,\delta})} \\
&\qquad \cdot \left|\left| \left(\mathcal{T}_{\sigma,\delta}- \widetilde{\mathcal{R}}_{\sigma,\delta}\right)\right|\right|_{BV(J_{\sigma,\delta})\to BV(J_{\sigma,\delta})}
\left|\left|q_{\sigma,\delta}\right|\right|_{BV(J_{\sigma,\delta})}.
\end{aligned}
\end{equation}
\end{proposicao}
\begin{proof}
Observe  that
\begin{align*}
\left(I-\mathcal{T}_{\sigma,\delta}\right)
\left(t_{\sigma,\delta}-q_{\sigma,\delta}\right)
&= t_{\sigma,\delta}-t_{\sigma,\delta}-q_{\sigma,\delta}
   +\mathcal{T}_{\sigma,\delta}q_{\sigma,\delta} \\
&= \left(\mathcal{T}_{\sigma,\delta}
   -\widetilde{\mathcal{R}}_{\sigma,\delta}\right)q_{\sigma,\delta}.
\end{align*}
Note that $\left(\mathcal{T}_{\sigma,\delta}- \widetilde{\mathcal{R}}_{\sigma,\delta}\right)(q_{\sigma,\delta})$ has zero mean so I can apply  $\left(I-\mathcal{T}_{\sigma,\delta}\right)^{-1}$ to both sides and obtain \eqref{BVtoBVest}.
\end{proof}
In the proof of the first item of Theorem~\ref{thm1}, we apply the
above proposition with
$\mathcal{T}_{\sigma,\delta}=\mathcal{L}_{\sigma,\delta}$.
For the second item, we choose $\mathcal{T}_{\sigma,\delta}$ to be
the transfer operator of an induced system whose unique fixed point
is the function $u_{\sigma,\delta}$ appearing in \eqref{cooleq}.
By \eqref{BVtoBVest}, for any of these choices of $\mathcal{T}_{\sigma,\delta},$  to prove Theorem~\ref{thm1} we need to estimate three quantities:
\begin{itemize}
\item[(1)] the norm of the inverse of $I-\mathcal{T}_{\sigma,\delta}$ restricted to the functions of mean $0$, i.e.
\begin{equation}\label{a1}
a_1(\mathcal{T}_{\sigma,\delta}):=  \left|\left|\left(I-\mathcal{T}_{\sigma,\delta}\right)^{-1}_{|BV_0(J_{\sigma,\delta})} \right|\right|_{BV_0(J_{\sigma,\delta}) \to BV_0(J_{\sigma,\delta})}
\end{equation}
\item[(2)] The $BV(J_{\sigma,\delta}) \to BV(J_{\sigma,\delta})$ norm of the difference operator, i.e. 
\begin{equation}\label{a2}
a_2(\mathcal{T}_{\sigma,\delta}):= \,\left|    \left| \left(\mathcal{T}_{\sigma,\delta}- \widetilde{\mathcal{R}}_{\sigma,\delta}\right)\right|\right|_{BV(J_{\sigma,\delta})\to BV(J_{\sigma,\delta}) },
\end{equation}
\item[(3)]  the BV norm of the quasistationary measure
\begin{equation}\label{a3}
a_3(\mathcal{T}_{\sigma,\delta}):=.  |q_{\sigma,\delta}|_{BV(J_{\sigma,\delta})}
\end{equation}
\end{itemize} 
A uniform estimate for \eqref{a3} follows directly from
\eqref{qsmbv} in Proposition~\ref{qsm}. The quantities in
\eqref{a1} and \eqref{a2} will be estimated in the proof of
Theorem~\ref{thm1}.

In the remainder of this section, we provide the tools needed to
estimate \eqref{a1}. We consider the operators
$\mathcal{T}_{\sigma,\delta}$ as a family
$\{\mathcal{T}_{\sigma,\delta}\}_{\sigma\leq\sigma_0,\,
\delta\leq\delta_0}$, where $\sigma_0$ and $\delta_0$ are
sufficiently small. We begin with the following well-known
proposition, which states that the set of invertible operators is open.
\begin{proposicao}\label{meanzero}
Let $\{\mathcal{U}_{\sigma,\delta}\}_{\sigma \le \sigma_0,\delta\le \delta_0}$ be a family of operators $ \mathcal{U}_{\sigma,\delta} \colon BV(J_{\sigma,\delta}) \to BV(J_{\sigma,\delta})$. Suppose that there exists a constant $C \ge 1$  such that
\begin{equation}\label{b1}
 \sup_{\sigma \le \sigma_0,\delta \le \delta_0} \left|\left|\left(I-\mathcal{U}_{\sigma,\delta}\right)^{-1}_{|BV_0(J_{\sigma,\delta})} \right|\right|_{BV_0(J_{\sigma,\delta})\to BV_0(J_{\sigma,\delta})}   < C \qquad 
\end{equation}
For given $\sigma,\delta$  let  
\begin{align*}
K(\sigma,\delta):= \left|\left|\mathcal{T}_{\sigma,\delta} -\mathcal{U}_{\sigma,\delta}\right| \right|_{BV(J_{\sigma,\delta}) \to BV(J_{\sigma,\delta})}
\end{align*}
If $ K(\sigma,\delta)< \frac{1}{C} $   then,  
\begin{equation}\label{newman}
\left|\left|\left(I-\mathcal{T}_{\sigma,\delta}\right)^-1_{BV_0(J_{\sigma,\delta})}  \right| \right|_{BV_0(J_{\sigma,\delta}) \to BV_0(J_{\sigma,\delta}) } \le\frac{1}{1-K(\sigma,\delta)C} C.
\end{equation}
\end{proposicao}
We omit the proof as it is well known.

As a consequence of the above Proposition, in order to find an estimate for $a_1(\mathcal{T}_{\sigma,\delta})$, it is sufficient to find a family of operators $\{\mathcal{U}_{\sigma,\delta}\}_{\sigma \le \sigma_0,\delta \le \delta_0}$ such that $\mathcal{U}_{\sigma,\delta}$  is close in $BV(J_{\sigma,\delta})$ norm to $\mathcal{T}_{\sigma,\delta}$ and  for which it is possible to check the uniform estimate in \eqref{b1}. 
A criterion introduced in \cite{galatolo2022statistical}, which we will now describe, gives sufficient conditions on $\{\mathcal{U}_{\sigma,\delta}\}_{\sigma \le \sigma_0,\delta \le \delta_0}$ to guarantee that \eqref{b1} is satisfied for $\sigma_0,\delta_0$ small enough.
\begin{definicao}\label{upf}
Let $\mathcal J\subseteq(0,\sigma_0]\times(0,\delta_0]$ have $(0,0)$
as an accumulation point. We say that the family of operators
$\{\mathcal{U}_{\sigma,\delta}\}_{(\sigma,\delta)\in\mathcal J}$
is a uniform family of operators if the following conditions are
satisfied for all sufficiently small $(\sigma,\delta)\in\mathcal J$:
\begin{itemize}
\item There exists an operator
$\mathcal{U}_0\colon BV(J_{\sigma,\delta}) \to BV(J_{\sigma,\delta})$
such that
\begin{equation}\label{uno}
||\mathcal{U}_{\sigma,\delta}-\mathcal{U}_0||_{BV(J_{\sigma,\delta}) \to L^1(J_{\sigma,\delta})}
= O(\sigma+\delta),
\end{equation}
\item There exist $A,B \ge 1$, $\eta<1$ such that, for all $n \ge 0$,
one has
\begin{equation}\label{dos}
\text{Var}_{J_{\sigma,\delta}}\left(\mathcal{U}_{\sigma,\delta}^n g\right)
\le A \eta^n \text{Var}_{J_{\sigma,\delta}}(g)
+ B|g|_{L^1(J_{\sigma,\delta})}
\qquad \forall g \in BV(J_{\sigma,\delta}),
\end{equation}
and
\begin{equation}\label{dos2}
\text{Var}_{J_{\sigma,\delta}}\left(\mathcal{U}_0^n g\right)
\le A \eta^n \text{Var}_{J_{\sigma,\delta}}(g)
+ B|g|_{L^1(J_{\sigma,\delta})}
\qquad \forall g \in BV(J_{\sigma,\delta}).
\end{equation}

\item There exist $C \ge 1$, $\gamma <1$ such that, if
$g \in BV_0(J_{\sigma,\delta})$, then
\begin{equation}\label{tres}
|\mathcal{U}_0^n(g)|_{L^1(J_{\sigma,\delta})}
\le C\gamma^n |g|_{BV(J_{\sigma,\delta})}.
\end{equation}
\item The following identity is satisfied:
\begin{equation}\label{cuatro}
\int_{J_{\sigma,\delta}} \mathcal{U}_{0}(\phi) dx
= \int_{J_{\sigma,\delta}}\mathcal{U}_{\sigma,\delta}(\phi) dx
= \int_{J_{\sigma,\delta}}\phi dx.
\end{equation}
\end{itemize}
\end{definicao}

\begin{teorema}\label{galato}
Suppose that the family
$\{\mathcal{U}_{\sigma,\delta}\}_{(\sigma,\delta)\in\mathcal J}$
is a uniform family of operators. Then there exist $A \ge 1$ and
$\gamma < 1$, independent of $(\sigma,\delta)$, such that, for all
sufficiently small $(\sigma,\delta)\in\mathcal J$ and for all
$\phi\in BV_0(J_{\sigma,\delta})$,
\begin{equation}\label{uniformNewmann}
||\mathcal{U}_{\sigma,\delta}^n(\phi)||_{BV(J_{\sigma,\delta})}
\le A\gamma^n ||\phi||_{BV(J_{\sigma,\delta})}.
\end{equation}
\end{teorema}
In particular, if
$\{\mathcal{U}_{\sigma,\delta}\}_{(\sigma,\delta)\in\mathcal J}$
is a uniform family of operators, then the uniform estimate in
\eqref{b1} holds for all sufficiently small
$(\sigma,\delta)\in\mathcal J$.

\section{Proof of Theorem \ref{thm1}}
\subsection{Proof of item 1}

First, we prove \eqref{fixedpest} in item $1$. The idea is to use
Proposition \ref{Bvdiff} choosing
$\mathcal{T}_{\sigma,\delta} = \mathcal{L}_{\sigma,\delta}$ and
$t_{\sigma,\delta}= \rho_{\sigma,\delta}$, which is the density of
the unique ergodic measure as stated in Hypothesis \ref{(A)}.
In particular, we need to estimate
$a_1(\mathcal{L}_{\sigma,\delta})$,
$a_2(\mathcal{L}_{\sigma,\delta})$ and
$a_{3}\left(\mathcal{L}_{\sigma,\delta}\right)$
(see equation \eqref{a1}-\eqref{a3} for the definition of these
objects). Because of Proposition \ref{qsm} and equation
\eqref{qsmbv}, there exists a constant $C > 0$ such that
\begin{equation}\label{c1}
a_3(\mathcal{L}_{\sigma,\delta}) < C
\qquad \text{as}\,\, \sigma,\delta \to 0.
\end{equation}
Estimates for $a_2$ are given in the following Proposition.
\begin{proposicao}\label{finalestimateeee}
The following estimate holds:
\begin{equation}\label{fixedpointestimate}
\left|\left|
\left(\mathcal{L}_{\sigma,\delta}
- \widetilde{\mathcal{R}}_{\sigma,\delta}\right)
\right|\right|_{BV(J_{\sigma,\delta}) \to BV(J_{\sigma,\delta})}
= O\left( \frac{\delta}{2\sigma} + \delta\right).
\end{equation}
\end{proposicao}
\begin{proof}
First observe that
\begin{align*}
&\left|\left|
\left(\mathcal{L}_{\sigma,\delta}
- \widetilde{\mathcal{R}}_{\sigma,\delta}\right)
\right|\right|_{BV(J_{\sigma,\delta}) \to BV(J_{\sigma,\delta})} \\
&\le
\left|\left|
\left(\mathcal{R}_{\sigma,\delta}
- \widetilde{\mathcal{R}}_{\sigma,\delta}\right)
\right|\right|_{BV(J_{\sigma,\delta}) \to BV(J_{\sigma,\delta})}
+
\left|\left|
\left(\mathcal{L}_{\sigma,\delta}
- \mathcal{R}_{\sigma,\delta}\right)
\right|\right|_{BV(J_{\sigma,\delta}) \to BV(J_{\sigma,\delta})} \\
&\le C\delta
\left|\left| \mathcal{R}_{\sigma,\delta}
\right|\right|_{BV(J_{\sigma,\delta}) \to BV(J_{\sigma,\delta})}
+
\left|\left|
\left(\mathcal{L}_{\sigma,\delta}
- \mathcal{R}_{\sigma,\delta}\right)
\right|\right|_{BV(J_{\sigma,\delta}) \to BV(J_{\sigma,\delta})} \\
&\le C\delta +
\left|\left|
\left(\mathcal{L}_{\sigma,\delta}
- \mathcal{R}_{\sigma,\delta}\right)
\right|\right|_{BV(J_{\sigma,\delta}) \to BV(J_{\sigma,\delta})},
\end{align*}
where from the first to the second inequality we used the
Lasota-Yorke estimate in \eqref{Lasota} and the estimate on the
spectral radius $\lambda_{\sigma,\delta}$ in
\eqref{spectralradiusstability}.
For the second term, note that we can write
\begin{align*}
\mathcal{L}_{\sigma,\delta}(\phi)(x)
= \frac{1}{2\sigma}\int_{-\sigma}^{\sigma}
\mathcal{L}_{\delta}\phi(x-u)\,du.
\end{align*}
As a consequence, if $\psi \in \mathcal{C}^1(J_{\sigma,\delta})$
has compact support in the interior of $J_{\sigma,\delta}$, we have
\begin{align*}
\left|
\int_{J_{\sigma,\delta}}
\mathcal{L}_{\sigma,\delta}(\phi)(x)\psi'(x)\,dx
\right| 
&=
\left|
\int_{J_{\sigma,\delta}}\int_{J_{\sigma,\delta}}
\frac{1_{B_{\sigma}(x)}(u)}{2\sigma}
\mathcal{L}_{\delta}(\phi)(u)\psi'(x)\,dx\,du
\right| \\
&=
\left|
\int_{J_{\sigma,\delta}}
\left[
\int_{J_{\sigma,\delta}}
\frac{1}{2\sigma}1_{B_{\sigma}(u)}(x)\psi'(x)\,dx
\right]
\mathcal{L}_{\delta}(\phi)(u)\,du
\right| \\
&\le
\frac{1}{\sigma}
|\psi|_{L^\infty(J_{\sigma,\delta})}
|\phi|_{L^1(J_{\sigma,\delta})}.
\end{align*}
Taking the supremum over such functions $\psi$ with
$|\psi|_{L^\infty(J_{\sigma,\delta})}\le 1$, we obtain
\begin{align*}
\operatorname{Var}_{J_{\sigma,\delta}}
\left(\mathcal{L}_{\sigma,\delta}(\phi)\right)
\le
\frac{1}{\sigma}|\phi|_{L^1(J_{\sigma,\delta})}.
\end{align*}
From this we deduce that
\begin{equation}\label{noisecompactness}
\|\mathcal{L}_{\sigma,\delta}(\phi)\|_{BV(J_{\sigma,\delta})}
\le
\left(\frac{1}{\sigma}+1\right)
|\phi|_{L^1(J_{\sigma,\delta})}
\le
\frac{2}{\sigma}|\phi|_{L^1(J_{\sigma,\delta})}.
\end{equation}
Now observe that
\begin{align*}
&\|\mathcal{L}_{\sigma,\delta}(\phi)
-\mathcal{R}_{\sigma,\delta}(\phi)\|_{BV(J_{\sigma,\delta})} \\
&=
\|\mathcal{L}_{\sigma,\delta}(\phi 1_{H_\delta})
\|_{BV(J_{\sigma,\delta})}\\
&\le
\frac{2}{\sigma}
|\phi 1_{H_\delta}|_{L^1(J_{\sigma,\delta})}\\
&\le
\frac{C}{\sigma}
|\phi|_{BV(J_{\sigma,\delta})}m(H_\delta),
\end{align*}
which implies that
\begin{align*}
\|\mathcal{L}_{\sigma,\delta}
-\mathcal{R}_{\sigma,\delta}\|
_{BV(J_{\sigma,\delta})\to BV(J_{\sigma,\delta})}
\le \frac{C\delta}{\sigma},
\end{align*}
which concludes the proof.
\end{proof}

It remains to estimate $a_1(\mathcal{L}_{\sigma,\delta})$.
The idea is to find a uniform family of operators, as in
Definition~\ref{upf}, and then apply Proposition~\ref{meanzero}.
To this end, let $\mathcal{L}_{\sigma}$ be the transfer operator
associated with the random perturbation of $f$, that is, with
$f_{\sigma,0}$. It follows from Hypothesis~\ref{(Z)},
\cite[Example~3.1]{Liverani2004InvariantMA} and
\cite[Theorem~30]{galatolo2022statistical} that the family $\{\mathcal{L}_{\sigma}\}_{\sigma\leq\sigma_0,\,
\delta\leq\delta_0}$ is a uniform
family of operators for $\sigma_0,\delta_0$ small enough.
Furthermore, the proof of Proposition~\ref{finalestimateeee} shows
that
\begin{align*}
\|\mathcal{L}_{\sigma}-\mathcal{R}_{\sigma,\delta}\|
_{BV(J_{\sigma,\delta})\to BV(J_{\sigma,\delta})}
= O\left(\frac{\delta}{\sigma}+\delta\right).
\end{align*}
As a consequence,
\begin{align*}
&\|\mathcal{L}_{\sigma}-\mathcal{L}_{\sigma,\delta}\|
_{BV(J_{\sigma,\delta})\to BV(J_{\sigma,\delta})} \\
&\leq
\|\mathcal{L}_{\sigma,\delta}-\mathcal{R}_{\sigma,\delta}\|
_{BV(J_{\sigma,\delta})\to BV(J_{\sigma,\delta})}
+
\|\mathcal{L}_{\sigma}-\mathcal{R}_{\sigma,\delta}\|
_{BV(J_{\sigma,\delta})\to BV(J_{\sigma,\delta})} \\
&= O\left(\frac{\delta}{\sigma}+\delta\right).
\end{align*}
Combining Theorem~\ref{galato} and Proposition~\ref{meanzero}, we
see that there exists a constant $C>0$ such that
\begin{align*}
a_1\left(\mathcal{L}_{\sigma,\delta}\right)<C
\qquad \text{as}\,\,\, \delta,\frac{\delta}{\sigma}\to 0.
\end{align*}
The above, combined with \eqref{c1}, \eqref{fixedpointestimate}
and \eqref{BVtoBVest}, proves \eqref{fixedpest}.
\subsection{Proof of item 2}
We consider the operator 
\begin{align*}
\mathcal{Q}_{\sigma,\delta}(\phi)
=
\mathcal{L}^{k_{\sigma,\delta}}_{\sigma,\delta}
(1_{H_\delta} \phi)
+
\mathcal{R}_{\sigma,\delta}
(1_{J_{\sigma,\delta}\setminus H_\delta}\phi).
\end{align*}
We are going to prove the following fact about $\mathcal{Q}_{\sigma,\delta}$.
\begin{lema}\label{lemma1}
The operator
$\mathcal{Q}_{\sigma,\delta}\colon BV(J_{\sigma,\delta}) \to BV(J_{\sigma,\delta})$
has unique fixed point $u_{\sigma,\delta}$.
Furthermore, one has that 
\begin{equation}\label{Blumenthalidea}
\|\mathcal{Q}_{\sigma,\delta}-\mathcal{R}_{\sigma,\delta}
\|_{BV(J_{\sigma,\delta}) \to BV(J_{\sigma,\delta})}
\le C\left(
\|f\|_{\mathcal{C}^2}\frac{1}{1-\gamma}
\delta |k_{\sigma,\delta}|
+
\frac{\lambda^{-(k_{\sigma,\delta}-1)}\delta}{2\sigma}
\right).
\end{equation}
and that there exists $C >0$ such that
\begin{equation}\label{uniformity}
\left\|
\left(I-\mathcal{Q}_{\sigma,\delta}\right)^{-1}_{|BV_0(J_{\sigma,\delta})}
\right\|_{BV_0(J_{\sigma,\delta}) \to BV_0(J_{\sigma,\delta})}
< C
\qquad
\text{as}\,\, \delta,\,
\frac{\lambda^{-(k_{\sigma,\delta}-1)}\delta}{\sigma}\to 0.
\end{equation}
\end{lema}
We claim that the above lemma implies item $2$ of Theorem \ref{thm1}.
In order to do that, we need to prove equations \eqref{cooleq} and
\eqref{cooleq2}. Using \eqref{Blumenthalidea}, \eqref{uniformity},
Proposition \ref{Bvdiff}, and the fact, proved at the beginning of
this section, that
\begin{align*}
\|\mathcal{Q}_{\sigma,\delta}
-\widetilde{\mathcal{R}}_{\sigma,\delta}
\|_{BV(J_{\sigma,\delta}) \to BV(J_{\sigma,\delta})}
\le C\delta
+
\|\mathcal{Q}_{\sigma,\delta}
-\mathcal{R}_{\sigma,\delta}
\|_{BV(J_{\sigma,\delta}) \to BV(J_{\sigma,\delta})},
\end{align*}
we have
\begin{align*}
&\|q_{\sigma,\delta}-u_{\sigma,\delta}\|_{BV(J_{\sigma,\delta})}
=
O\left(
\delta k_{\sigma,\delta}
+
\frac{\lambda^{-(k_{\sigma,\delta}-1)}\delta}{\sigma}
\right),\\
&\hspace{4cm}
\text{as}\,\, \delta,\,
\frac{\lambda^{-(k_{\sigma,\delta}-1)}\delta}{\sigma}\to 0,
\end{align*}
i.e. the estimate \eqref{cooleq2} in item $2$ holds. Furthermore,
one can check that
\begin{align*}
h_{\sigma,\delta}
=
u_{\sigma,\delta}
+
\sum_{j=1}^{k_{\sigma,\delta}-1}
\mathcal{L}^j_{\sigma,\delta}
(1_{H_{\delta}} u_{\sigma,\delta})
\in BV(J_{\sigma,\delta})
\end{align*}
and is a fixed point for $\mathcal{L}_{\sigma,\delta}$, hence
$h_{\sigma,\delta}$ is proportional to $\rho_{\sigma,\delta}$.
Normalizing $h_{\sigma,\delta}$ to have integral one, we have
\begin{align*}
\rho_{\sigma,\delta}
=
\frac{
u_{\sigma,\delta}
+
\sum_{j=1}^{k_{\sigma,\delta}-1}
\mathcal{L}^j_{\sigma,\delta}
(1_{H_{\delta}} u_{\sigma,\delta})
}{
1 +
\left|
\sum_{j=1}^{k_{\sigma,\delta}-1}
\mathcal{L}^j_{\sigma,\delta}
(1_{H_{\delta}}u_{\sigma,\delta})
\right|_{L^1(J_{\sigma,\delta})}
},
\end{align*}
where
\begin{align*}
\left|
\sum_{j=1}^{k_{\sigma,\delta}-1}
\mathcal{L}^j_{\sigma,\delta}
(1_{H_{\delta}}u_{\sigma,\delta})
\right|_{L^1(J_{\sigma,\delta})}
&\le C(k_{\sigma,\delta}-1)\delta
|u_{\sigma,\delta}|_{BV(J_{\sigma,\delta})} \\
&\le C\delta \log(\delta),
\qquad
\text{as} \,\, \delta,\,
\frac{\lambda^{-(k_{\sigma,\delta}-1)}\delta}{\sigma}\to 0,
\end{align*}
since $k_{\sigma,\delta}$ is capped to $\log(\delta).$
This shows that equation \eqref{cooleq} is satisfied and concludes
the proof of Theorem \ref{thm1}. \\
It remains to prove Lemma \ref{lemma1}.
\begin{proof}
Note that, if $\phi \in L^1(J_{\sigma,\delta})$, then
\begin{align*}
\text{Var}_{J_{\sigma,\delta}}
\left(\mathcal{Q}_{\sigma,\delta}(\phi)\right)
&\le
\frac{1}{\sigma}
\left|
\mathcal{L}_{\sigma,\delta}^{k_{\sigma,\delta}-1}
(1_{H_\delta}\phi)
\right|_{L^1(J_{\sigma,\delta})}
+
\frac{1}{\sigma}
|1_{J_{\sigma,\delta}\setminus H_\delta}\phi|
_{L^1(J_{\sigma,\delta})}\\
&\le \frac{1}{\sigma}|\phi|_{L^1(J_{\sigma,\delta})}.
\end{align*}
As a consequence, by Helly's theorem there is at least one fixed point
$u_{\sigma,\delta} \in BV(J_{\sigma,\delta})$.
Suppose there are two fixed points $u_{\sigma,\delta}$ and
$v_{\sigma,\delta}$. By uniqueness of $\rho_{\sigma,\delta}$, the
following functions are proportional:
\begin{align*}
u_{\sigma,\delta}
+
\sum_{j=1}^{k_{\sigma,\delta}-1}
\mathcal{L}^j_{\sigma,\delta}(1_{H_{\delta}}u_{\sigma,\delta})
\quad\text{and}\quad
v_{\sigma,\delta}
+
\sum_{j=1}^{k_{\sigma,\delta}-1}
\mathcal{L}^j_{\sigma,\delta}(1_{H_{\delta}}v_{\sigma,\delta}).
\end{align*}
By definition of $k_{\sigma,\delta}$,
\begin{align*}
&\sum_{j=1}^{k_{\sigma,\delta}-1}
\mathcal{L}^j_{\sigma,\delta}
(1_{H_{\delta}}u_{\sigma,\delta})(x)
=
\sum_{j=1}^{k_{\sigma,\delta}-1}
\mathcal{L}^j_{\sigma,\delta}
(1_{H_{\delta}}v_{\sigma,\delta})(x)
=0
\end{align*}
for almost every $x\in H_\delta$. Hence
$u_{\sigma,\delta}$ and $v_{\sigma,\delta}$ are proportional on
$H_\delta$. It follows that they are proportional on
$J_{\sigma,\delta}$, and their normalization gives
$u_{\sigma,\delta}=v_{\sigma,\delta}$. \\

Now we prove equation \eqref{Blumenthalidea}. Observe that
\begin{align*}
&\text{Var}_{J_{\sigma,\delta}}
\left((\mathcal{Q}_{\sigma,\delta}
-\mathcal{R}_{\sigma,\delta})\phi\right) \\
&\le
\text{Var}_{J_{\sigma,\delta}}
\left(\mathcal{L}^{k_{\sigma,\delta}}_{\sigma,\delta}
\left(1_{H_\delta} \phi \right)\right)\\
&\le
\sup_{\omega \in \Omega}
\text{Var}_{J_{\sigma,\delta}}
\left(
\frac{\mathcal{L}_{\sigma,\delta}(1_{H_\delta}\phi)}
{|df^{k_{\sigma,\delta}-1}_{\theta(\omega)}|}
\circ
\left(f^{k_{\sigma,\delta}-1}_{\theta(\omega)}\right)^{-1}
\right) \\
&\le
\sup_{\omega \in \Omega}
\text{Var}_{[g(H_{\delta})-\sigma,g(H_{\delta})+\sigma]}
\left(
\frac{\mathcal{L}_{\sigma,\delta}(1_{H_\delta}\phi)}
{|df^{k_{\sigma,\delta}-1}_{\theta(\omega)}|}
\right)
+
\sup_{J_{\sigma,\delta}}
\left|\mathcal{L}_{\sigma,\delta}(1_{H_\delta}\phi)\right|
\lambda^{-(k_{\sigma,\delta}-1)} \\
&\le
\lambda^{-(k_{\sigma,\delta}-1)}
\text{Var}_{J_{\sigma,\delta}}
\left(\mathcal{L}_{\sigma,\delta}(1_{H_\delta}\phi)\right)
+
D_{k_{\sigma,\delta}}
\int_{[g(H_{\delta})-\sigma,g(H_{\delta})+\sigma]}
\left|\mathcal{L}_{\sigma,\delta}
(\phi 1_{H_{\delta}})(x)\right|\,dx \\
&\quad+
\sup_{J_{\sigma,\delta}}
\left|\mathcal{L}_{\sigma,\delta}(1_{H_\delta}\phi)\right|
\lambda^{-(k_{\sigma,\delta}-1)} \\
&\le
C\left(
\frac{\lambda^{-(k_{\sigma,\delta}-1)}\delta}{\sigma}
+
D_{k_{\sigma,\delta}}\delta
\right)
\|\phi\|_{BV(J_{\sigma,\delta})},
\end{align*}
where
\begin{align*}
D_{k_{\sigma,\delta}}
:=
\sup_{\omega \in \Omega}
\sup_{x \in [g(H_{\delta})-\sigma,g(H_{\delta})+\sigma]}
\left|
D\left(
\frac{1}{|df^{k_{\sigma,\delta}-1}_{\omega}|}
\right)(x)
\right|.
\end{align*}
In order to estimate $D_{k_{\sigma,\delta}}$, note that
\begin{align*}
D(df^{k_{\sigma,\delta}-1}_{\omega})(x)
&=
\sum_{i=0}^{k_{\sigma,\delta}-2}
\left(\prod_{j=0}^{i-1}f'(f^j_{\omega}(x))\right)
f''(f^i_{\omega}(x))
\left(\prod_{j=i+1}^{k_{\sigma,\delta}-2}
f'(f^j_{\omega}(x))\right)
d(f^i_{\omega}(x)),
\end{align*}
with $M:= \sup_{x \in \mathbb{S}^1}|f''(x)|$.
Now observe that the sign of the derivative is constant on
$[g(H_{\delta})-\sigma,g(H_{\delta})+\sigma]$, so that
\begin{align*}
&\left|
D\left(
\frac{1}{|df^{k_{\sigma,\delta}-1}_{\omega}|}
\right)(x)
\right| \\
&\le M
\frac{
\sum_{i=0}^{k_{\sigma,\delta}-2}
\left|\prod_{\substack{0\le j\le k_{\sigma,\delta}-2\\j\neq i}}
f'(f^j_{\omega}(x))\right|
\left|d(f^i_{\omega}(x))\right|
}{
\prod_{j=0}^{k_{\sigma,\delta}-2}
|f'(f^j_{\omega}(x))|^2
}\\
&\le M
\sum_{i=0}^{k_{\sigma,\delta}-2}
\frac{1}{
|f'(f^i_\omega(x))|\,
|df^{k_{\sigma,\delta}-i-1}_{\theta^{i}(\omega)}
(f^i_{\omega}(x))|
} \\
&\le M\frac{1}{1-\gamma}.
\end{align*}
Naturally,
\begin{align*}
|(\mathcal{Q}_{\sigma,\delta}
-\mathcal{R}_{\sigma,\delta})\phi|_{L^1(J_{\sigma,\delta})}
\le |1_{H_\delta}\phi|_{L^1(J_{\sigma,\delta})}
\le C\delta|\phi|_{BV(J_{\sigma,\delta})},
\end{align*}
so that
\begin{align*}
\|\mathcal{R}_{\sigma,\delta}
-\mathcal{Q}_{\sigma,\delta}\|
_{BV(J_{\sigma,\delta}) \to BV(J_{\sigma,\delta})} \le C\left(
|f|_{\mathcal{C}^2}\frac{1}{1-\gamma}
\delta k_{\sigma,\delta}
+
\frac{\lambda^{-(k_{\sigma,\delta}-1)}\delta}{2\sigma}
\right).
\end{align*}
This proves \eqref{Blumenthalidea}.It remains to prove \eqref{uniformity}. For $\varepsilon>0$, let
\[
\mathcal J_\varepsilon
:=
\left\{
(\sigma,\delta)\in(0,\sigma_0]\times(0,\delta_0]
\,:\,
\delta k_{\sigma,\delta}
+\frac{\delta}{\sigma\lambda^{k_{\sigma,\delta}-1}}
\leq\varepsilon
\right\}.
\]
By Theorem~\ref{galato}, it suffices to show that, for
$\sigma_0,\delta_0$ and $\varepsilon$ sufficiently small,
$\{\mathcal Q_{\sigma,\delta}\}_{(\sigma,\delta)\in\mathcal J_\varepsilon}$
is a uniform family of operators in the sense of Definition~\ref{upf}.

Indeed, we have
\begin{align*}
&\left|
\mathcal{Q}_{\sigma,\delta}(\phi)
-\mathcal{L}(\phi)
\right|_{L^1(J_{\sigma,\delta})} \\
&\le
\left|
\mathcal{L}^{k_{\sigma,\delta}}_{\sigma,\delta}(1_{H_\delta}\phi)
-\mathcal{L}(1_{H_\delta}\phi)
\right|_{L^1(J_{\sigma,\delta})}
+
\left|
\left(\mathcal{L}_{\sigma}-\mathcal{L}\right)
(1_{J_{\sigma,\delta}\setminus H_\delta}\phi)
\right|_{L^1(J_{\sigma,\delta})}\\
&\le
4\delta|\phi|_{BV(J_{\sigma,\delta})}
+
\left\|
\mathcal{L}_{\sigma}-\mathcal{L}
\right\|_{BV(J_{\sigma,\delta})\to L^1(J_{\sigma,\delta})} \\
&\le
C(\delta+\sigma)|\phi|_{BV(J_{\sigma,\delta})},
\end{align*}
where the last line follows from Proposition
\cite[Example 3.1]{Liverani2004InvariantMA} choosing $\sigma_0$ and
$\delta_0$ small enough. This proves \eqref{uno}.

To obtain \eqref{dos}, first observe that a telescoping argument, along with 
\eqref{Blumenthalidea},gives
\begin{equation}\label{intheregime}
\|\mathcal Q_{\sigma,\delta}^N
-\mathcal R_{\sigma,\delta}^N\|
_{BV(J_{\sigma,\delta})\to BV(J_{\sigma,\delta})}
\le C_N\left(
\delta k_{\sigma,\delta}
+\frac{\lambda^{-(k_{\sigma,\delta}-1)}\delta}{\sigma}
\right),
\end{equation}
where $N$ is chosen large enough such that $A\gamma^N<1/4$ in
\eqref{Lasota}, so that in particular for every
$\phi\in BV(J_{\sigma,\delta})$,
\[
\operatorname{Var}_{J_{\sigma,\delta}}
\left(\mathcal R_{\sigma,\delta}^{N}\phi\right)
\leq
\frac14\operatorname{Var}_{J_{\sigma,\delta}}(\phi)
+B\|\phi\|_{L^1(J_{\sigma,\delta})}.
\]
 In the regime
\[
\left(
\delta k_{\sigma,\delta}
+\frac{\lambda^{-(k_{\sigma,\delta}-1)}\delta}{\sigma}
\right)\to 0
\]
the right-hand side of \eqref{intheregime} at most $1/4$. Then, we have
\[
\operatorname{Var}_{J_{\sigma,\delta}}
(\mathcal Q_{\sigma,\delta}^N\phi)
\le \tfrac12\operatorname{Var}_{J_{\sigma,\delta}}(\phi)
+(B+\tfrac14)\|\phi\|_{L^1(J_{\sigma,\delta})}.
\]
Since $\mathcal Q_{\sigma,\delta}$ is
$L^1(J_{\sigma,\delta})$-contractive, iteration in blocks of length
$N$, followed by the uniformly bounded remaining iterates, proves
\eqref{dos}. Equation \eqref{dos2} follows from the Lasota--Yorke
estimate for $\mathcal L$, equation \eqref{tres} follows from
\eqref{maybeused}, and \eqref{cuatro} follows from preservation of
mass. Theorem~\ref{galato} therefore gives a uniform bound on
$(I-\mathcal Q_{\sigma,\delta})^{-1}$ on
$BV_0(J_{\sigma,\delta})$.

\end{proof}

\section{Proof of Theorem \ref{thm2}}
\subsection{Proof of Theorem \ref{thm2} }
Let $f_{\sigma,\delta}$ be an admissible Markov chain. We first
prove that $f_{\sigma,\delta}$ satisfies
Hypotheses~\ref{(Z)}--\ref{(C)}, and then establish
\eqref{est1}, \eqref{est2} and \eqref{finalest}, respectively in
the cases $p_\delta=1$, $p_\delta=p>1$ independent of $\delta$,
and $p_\delta\asymp\log(1/\delta)$.

Since $f$ is a $C^2$ uniformly expanding circle map, it
has no turning points and its transfer operator has a spectral gap
on $BV(\mathbb S^1)$. Hence Hypothesis~\ref{(Z)} holds; see
\cite{galatolo2022statistical}. Hypothesis~\ref{(Z')} holds with
$J_{\sigma,\delta}=\mathbb S^1$. Moreover,
Hypothesis~\ref{(C)} holds because $f$ is $ C^2$.

It remains to prove the existence of a unique ergodic measure for $f_{\sigma,\delta}$, as required by Hypothesis~\ref{(A)}. 

We claim that $f_{\sigma,\delta}$ has unique ergodic measure full supported in $\mathbb{S}^1$ as long as $\sigma \lambda^{p_{\delta}-1} \ge 2\delta$. To prove such claim, we need to show that, for every $x \in \mathbb{S}^1$ and open set $A \subset \mathbb{S}^1$,  there exists $n>0$ such that 
\begin{align*}
\mathbb{P}^n_x(A)>0.
\end{align*}
So let us fix $x\in \mathbb{S}^1$ and an open set $A\subset \mathbb S^1.$ Let $x_0$, possibly depending on $\delta$, be the periodic sink of $f_\delta$ given by admissibility. Since $f$ is $C^2$ and uniformly expanding, there is $J\subset B_{\delta}(x_0),$ such that, for some $k\ge 0$
\begin{equation}\label{keyidea}
\begin{aligned}
f^k(J) &\subset A, \\f^i(J)&\cap B_{\delta}(x_0) = \emptyset \qquad i <k.    
\end{aligned}
\end{equation}
After one iterate, the set of points reachable from $x$ contains
$B_\sigma(f_\delta(x))$. If this ball intersects $A$, we are done.
Otherwise, we apply $f_\delta$ until an iterate intersects either
$A$ or $V_\delta$. If it intersects first $V_{\delta},$ then let $y\in V_{\delta}$ be any reachable point. Then, since $\sigma \lambda^{p_{\delta}-1} \ge 2\delta$, using \eqref{boundedaway} when $p_{\delta} \asymp -\log(\delta),$ we have
\begin{align*}
\bigcup_{\omega \in \Omega} f^{p_{\delta}}_{\omega,\delta}(y) \supset f_{\delta}^{p_{\delta}-1}B_{\sigma}(f(y))\supset B_{\sigma \lambda^{p_{\delta}-1}}(f^{p_{\delta}(y)})\supset B_{\delta}(x_0),
\end{align*}
and the claim follows by \eqref{keyidea}.

We now prove \eqref{est1}, \eqref{est2}, and \eqref{finalest}.
When $p_\delta=1$, we have $k_{\sigma,\delta}=1$ and
\begin{align*}
\int_{\mathbb{S}^1}
\log|df_{\delta}(x)|\rho_{\sigma,\delta}(x)\,dx
&=
\int_{B_{\delta}(x_0)}
\log|df_{\delta}(x)|\rho_{\sigma,\delta}(x)\,dx
+
\int_{\mathbb{S}^1\setminus B_{\delta}(x_0)}
\log|df(x)|\rho_{\sigma,\delta}(x)\,dx \\
&\ge
\int_{\mathbb{S}^1\setminus B_{\delta}(x_0)}
\log|df(x)|q_{\sigma,\delta}(x)\,dx
+O\left(\delta+\frac{\delta}{\sigma}\right) \\
&\quad-
\left(C+O\left(\delta+\frac{\delta}{\sigma}\right)\right)
\int_{B_{\delta}(x_0)}\log|df_{\delta}(x)|\,dx \\
&=
\xi_0
+O\left(\sigma|\log\sigma|+\delta |\log(\delta)|
+\frac{\delta}{\sigma}\right) \\
&\quad-
\left(C+O\left(\delta+\frac{\delta}{\sigma}\right)\right)
\int_{B_{\delta}(x_0)}\log|df_{\delta}(x)|\,dx.
\end{align*}
Here, in the inequality we used \eqref{fixedpest} and
\eqref{qsmbv}. The last equality follows from \eqref{convergence},
the boundedness of $\log|df|$, and \eqref{qsmbv}.
Equation \eqref{est1} follows by estimating the contribution
of $B_\delta(x_0)$ using the critical behaviour of $f_\delta$
in item~\ref{c} of Definition~\ref{admissible}.

When $p_\delta=p>1$ is independent of $\delta$, we have
$k_{\sigma,\delta}=p$ and
\begin{align*}
\int_{\mathbb{S}^1}
\log|df_{\delta}(x)|\rho_{\sigma,\delta}(x)\,dx
&=
\int_{B_{\delta}(x_0)}
\log|df_{\delta}(x)|\rho_{\sigma,\delta}(x)\,dx \\
&\quad+
\int_{\mathbb{S}^1}
\log|df(x)|\rho_{\sigma,\delta}(x)\,dx \\
&\quad-
\int_{B_{\delta}(x_0)}
\log|df(x)|\rho_{\sigma,\delta}(x)\,dx \\
&\ge
\int_{\mathbb{S}^1}
\log|df(x)|\rho_{\sigma,\delta}(x)\,dx \\
&\quad-
\left(C+O\left(\frac{\delta}{\sigma\lambda^{p-1}}\right)\right)
\int_{B_{\delta}(x_0)}\left|\log|df_{\delta}(x)|\right|\,dx \\
&\quad-
O\left(\delta+\frac{\delta^2}{\sigma\lambda^{p-1}}\right),
\end{align*}
where, in the inequality, we used \eqref{cooleq}, \eqref{cooleq2} and \eqref{qsmbv}.

The integral term in $B_{\delta}(x_0)$ is estimated as in
the case $p=1$. We now estimate the integral involving
$\log|df|$:
\begin{align*}
\int_{\mathbb S^1}\log|df(x)|\rho_{\sigma,\delta}(x)\,dx
&=
\frac{
\displaystyle\int_{\mathbb S^1}
\log|df(x)|u_{\sigma,\delta}(x)\,dx
}{
1+\|R_{\sigma,\delta}\|_{L^1}
}\\
&\quad+
\frac{
\displaystyle\int_{\mathbb S^1}
\log|df(x)|R_{\sigma,\delta}(x)\,dx
}{
1+\|R_{\sigma,\delta}\|_{L^1}
}\\
&\geq
\frac{1}{1+C\delta}
\int_{\mathbb S^1}\log|df(x)|q_{\sigma,\delta}(x)\,dx\\
&\quad-
\frac{
O\left(\delta+\frac{\delta}{\sigma\lambda^{p-1}}\right)
}{
1+C\delta
}\\
&\geq
\xi_0-O\left(
\delta+\frac{\delta}{\sigma\lambda^{p-1}}
+\sigma|\log\sigma|+\delta|\log\delta|
\right).
\end{align*}
Here, the equality follows from \eqref{cooleq}.
In the first inequality, we used \eqref{cooleq2} and the estimate $\|R_{\sigma,\delta}\|_{L^1}=O(\delta)$
in Theorem~\ref{thm1}, and the last inequality follows from \eqref{convergence}.

Suppose that $p_\delta\asymp\log(1/\delta)$ and choose
$\sigma=\delta/\log(1/\delta)$. By \eqref{boundedaway}, the orbit of
$x_0$ remains at distance at least $d_*$ from $x_0$ for every
$1\leq j<p_\delta$. Moreover, there exists $a>0$ such that, for all
sufficiently small $\delta$, the random orbit of every
$x\in B_\delta(x_0)$ remains within distance $d_*/4$ of the orbit of
$x_0$ for at least $a\log(1/\delta)$ iterates. Consequently,
\[
k_{\sigma,\delta}\geq a\log(1/\delta).
\]
Estimate \eqref{finalest} follows by substituting this bound into
\eqref{est2}.
\bibliographystyle{plain} 
\bibliography{biblio}
\end{document}